\documentclass[10pt]{amsart}
\usepackage{amssymb}
\usepackage{epsfig}
\usepackage{url}
\usepackage{setspace}
\theoremstyle{plain}

\newtheorem{thm}{Theorem}[section]
\newtheorem{cor}[thm]{Corollary}

\newtheorem{rem}[thm]{Remark}
\newtheorem{ques}[thm]{Question}
\newtheorem{conj}[thm]{Conjecture}
\newtheorem{exam}[thm]{Example}
\def\cal{\mathcal}
\def\bbb{\mathbb}
\def\op{\operatorname}
\renewcommand{\phi}{\varphi}
\newcommand{\R}{\bbb{R}}

\newcommand{\Z}{\bbb{Z}}
\newcommand{\Q}{\bbb{Q}}

\newcommand{\C}{\bbb{C}}

\begin{document}

\title{Rational points on certain quintic hypersurfaces}
\author{Maciej Ulas}


\begin{abstract}
Let $f(x)=x^5+ax^3+bx^2+cx \in \Z[x]$ and consider the hypersurface
of degree five given by the equation
\begin{equation*}
\cal{V}_{f}: f(p)+f(q)=f(r)+f(s).
\end{equation*}
Under the assumption $b\neq 0$ we show that there exists
$\Q$-unirational elliptic surface contained in $\cal{V}_{f}$. If
$b=0,\;a<0$ and $-a\not\equiv 2,18,34 \pmod {48}$ then there exists
$\Q$-rational surface contained in $\cal{V}_{f}$. Moreover, we prove
that for each $f$ of degree five there exists $\Q(i)$-rational
surface contained in $\cal{V}_{f}$.
\end{abstract}

\maketitle

\section{Introduction}\label{Section1}

In this paper we are interested in the problem of the existence of
integer and rational points on the hypersurface given by the
equation
\begin{equation*}
\cal{V}_{f}: f(p)+f(q)=f(r)+f(s),
\end{equation*}
where $f\in\Q[X]$ and $\op{deg}f=5$. Moreover, we assume that for
each pair $a,b\in\Q\setminus\{0\}$ we have $f(ax+b)\neq cx^5+d$ for
any $c,d\in\Q$. This assumption guarantees that $\cal{V}_{f}$ is an
affine algebraic variety of dimension three. The set of rational
points on $\cal{V}_{f}$ we denote by $\cal{V}_{f}(\Q)$. In other
words
\begin{equation*}
\cal{V}_{f}(\Q)=\{(p,q,r,s)\in\Q^{4}:\;f(p)+f(q)=f(r)+f(s)\}.
\end{equation*}
Similarly, by $\cal{V}_{f}(\Z)$ we denote the set of integer points
on $\cal{V}_{f}$, so $\cal{V}_{f}(\Z)=\cal{V}_{f}(\Q)\cap \Z^{4}.$

We will say that the point $P=(p,q,r,s)\in\cal{V}_{f}$ is
non-trivial if $\{p,q\}\cap\{r,s\}=\emptyset$ and
$\{f(p),f(q)\}\cap\{f(r),f(s)\}=\emptyset$. By $T_{f}$ we will be
denoted the set of trivial rational points on $\cal{V}_{f}$. Let us
note that each singular point is trivial and that the number of all
singular points (rational or not) is finite. In the sequel by
rational point we will understand non-trivial rational point.

The problem of the existence of integer points on the hypersurface
$\cal{V}_{f}$ was investigated in the interesting work of Browning
\cite{Brow}. In this work it was shown that
\begin{equation*}
M(f;B)\ll_{\epsilon,\;f}B^{1+\epsilon}(B^{1/3}+B^{2/\sqrt{5}+1/4}),
\end{equation*}
for each $\epsilon >0$. Here $M(f;B)$ is the number of solutions
$(p,q,r,s)$ of the equation which define $\cal{V}_{f}$ and with such
a property that $0<p,q,r,s\leq B$ and
$\{p,q\}\cap\{r,s\}=\emptyset$. From the above estimation we can see
that the set of positive integer points on $\cal{V}_{f}$  is rather
"thin". According to the best Author knowledge we do not know any
example of a polynomial $f$ of degree five with such a property that
the set $\cal{V}_{f}(\Z)\setminus T_{f}$ is infinite. Moreover, we
are unable to find in the existing literature of subject any example
of a polynomial $f$ of degree five which gives a positive answer to
the following:

\begin{ques}\label{ques1}
Let $N>1$ be given. Is it possible to construct a polynomial $f$ of
degree five, such that $\sharp (\cal{V}_{f}(\Z)\setminus T_{f})>N$?
\end{ques}

It is clear that the question concerning the existence of a
polynomial $f$ of degree five with such a property that the set
$\cal{V}_{f}(\Q)$ is infinite should be easier. So, it is natural to
state the following:

\begin{ques}\label{ques2}
For which polynomials $f$ of degree five the set $\cal{V}_{f}(\Q)$
is infinite?
\end{ques}

It seems that these questions have not been considered earlier. It
is also clear that in the case of Question \ref{ques2}  we can
consider polynomials of the form $f(X)=X^5+aX^3+bX^2+cX$ only, where
$a,b,c\in\Z$ and at least one among the numbers $a,b,c$ is nonzero.
We will see that if $b\neq 0$ then the diophantine equation
$f(p)+f(q)=f(r)+f(s)$ has rational two parametric solution (Theorem
\ref{thm1}). In geometrical terms this means that the there is a
rational surface contained in $\cal{V}_{f}$. From this result we can
deduce easily that the answer for the Question \ref{ques1} is
positive. Moreover, we will prove that for any polynomial $f$ of
degree five there exists $\Q(i)$-rational surface contained in
$\cal{V}_{f}$ (Theorem \ref{thm2}).

\section{Construction of rational points on the $\cal{V}_{f}$}\label{Section2}

Let $f\in\Q[X]$ and suppose that $\op{deg}f=5$. In this section we
will show how we can construct parametric solutions of the equation
which define the hypersurface
\begin{equation*}
\cal{V}_{f}:f(p)+f(q)=f(r)+f(s).
\end{equation*}

Because we are interested in rational solutions, so without loss of
generality we can assume that $f(X)=X^5+aX^3+bX^2+cX,\;a,b,c\in\Z$
and at least one among the numbers $a,b,c$ is nonzero.

Our aim is the proof of the following theorem.
\begin{thm}\label{thm1}
Let $f(X)=X^5+aX^3+bX^2+cX\in\Z[X]$, where $b\neq 0$ and consider
the hypersurface $\cal{V}_{f}$. Then, there exists $\Q$-unirational
elliptic surface $\cal{E}_{f}$ such that
$\cal{E}_{f}(\Q)\subset\cal{V}_{f}(\Q)$. In particular, the set
$\cal{V}_{f}(\Q)$ is infinite.
\end{thm}
\begin{proof}
In the equation which define $\cal{V}_{f}$ we make (non-invertible)
change of variables given by
\begin{equation}\label{R1}
p=x,\quad q=y-x,\quad r=z,\quad s=y-z.
\end{equation}
The result of this substitution is the following
\begin{equation*}
f(x)+f(y-x)-f(z)-f(y-z)=(x-z)(x-y+z)G(x,y,z),
\end{equation*}
where $G(x,y,z)=2b+3ay+5x^2y-5xy^2+5y^3-5y^2z+5yz^2$. From the
geometric point of view the substitution we have used can be
understood as the cross section of the hypersurface $\cal{V}_{f}$
with the hyperplane $L$ given by the equation $L: p+q=r+s$ (the
system of equations (\ref{R1}) give the parametrization of $L$).

Let us note that the equation $G(x,y,z)=0$ has a solution in
rational numbers if and only if the discriminant of the polynomial
$G$ in respect to $z$ is a square of a rational number, say $v$. So,
we are interested in the construction of rational points on the
surface
\begin{equation*}
\cal{S}: v^2=-5y(15y^3+20xy(x-y)+12ay+8b)=:\Delta(x,y).
\end{equation*}

If we make a change of variables
\begin{equation*}
(x,y,w)=\Big(-\frac{5b(t+1)}{X+5a},-\frac{10b}{X+5a},\frac{20bY}{(X+5a)^2}\Big),
\end{equation*}
with the inverse
\begin{equation*}
(X,t,Y)=\Big(-\frac{5(2b+ay)}{y},\frac{2x-y}{y},\frac{5bw}{y^2}\Big)
\end{equation*}
the surface $\cal{S}$ is transformed to the form
\begin{equation*}
\cal{E}: Y^2=X^3-75a^2X-125(5bt^2+10b^2+2a^3).
\end{equation*}

Let us note that the surface $\cal{E}$ is of degree three and
contains rational curve at infinity $[X:Y:t:Z]=[0:1:t:0]$. So we can
invoke Segre theorem which says that the surface of degree three
with a rational point is unirational. In other words there is a
rational function
\begin{equation*}
\Phi:\Q\times\Q\ni(u,v)\mapsto \Phi(u,v)\in\cal{E}_{f},
\end{equation*}
such that the set $\Phi(\Q\times\Q)$ is dense (in Zariski topology)
in the set $\cal{E}_{f}(\C)$.

For convenience of the reader we will show how the function $\Phi$
can be constructed.

Let us put $F(X,Y,t)=Y^2-(X^3-75a^2X-125(5bt^2+10b^2+2a^3))$. We use
the method of indetermined coefficients in order to find
two-parametric solution of  $F(X,Y,t)=0$. Let $u,v$ be parameters
and let us put
\begin{equation}\label{R2}
X=T^2+10uT+p,\quad Y=T^3+qT^2+rT,\quad t=(v/5b)T^2+s.
\end{equation}
We want to find $p,q,r,s,T\in\Q(u,v)$ with such a property that the
equation $F(X,Y,t)=0$ is satisfied identicaly. For the quantities
given by (\ref{R2}) we have
\begin{equation*}
F(X,Y,t)=a_{0}+a_{1}T+a_{2}T^2+a_{3}T^3+a_{4}T^4+a_{5}T^5,
\end{equation*}
where
\begin{equation*}
\begin{array}{lll}
  a_{0}=250a^3+1250b^2+75a^2p-p^3+625b^2s^2, &\quad a_{1}= 30(5a-p)(5a+p)u, \\
  a_{2}=75a^2-3p^2+r^2-300pu^2+250bsv, &\quad a_{3}=2(qr-30pu-500u^3), \\
  a_{4}=-3p+q^2+2r-300u^2+25v^2, &\quad a_{5}= 2(q-15u).\\
\end{array}
\end{equation*}
Let us notice that the system of equations
$a_{2}=a_{3}=a_{4}=a_{5}=0$ has exactly one solution in $\Q(u,v)$
given by
\begin{equation}\label{R3}
\begin{array}{ll}
  p=25(u^2-3v^2)/3, &\quad q=15u, \\
  r=50(u^2-v^2), & \quad s=(25u^4-450u^2v^2-75v^4-9a^2)/30bv. \\
\end{array}
\end{equation}

We can see that if $p,q,r,s$ are given by (\ref{R3}) then
$F(T^2+10uT+p,T^3+qT^2+rT,(v/5b)T^2+s)\in\Q(u,v)[T]$ and
$\op{deg}_{T}F=1$. So this polynomial has a root in the field
$\Q(u,v)$ of the form
\begin{equation*}
T=-\frac{250a^3+1250b^2+75a^2p-p^3+625b^2s^2}{30(5a-p)(5a+p)u},
\end{equation*}
where $p,q,r,s$ are given by (\ref{R3}). Putting the calculated
values $p,q,r,s,T$ to the equations (\ref{R2}) we get the solutions
dependent on two parameters $u,\;v$ we are looking for.

\end{proof}

\begin{rem}
{\rm It should be noted that the same method was used by Whitehead
\cite{Whi} in order to give the proof of unirationality of the
surface $z^2=h(x,y)$, where $h\in\Q(x,y)$ is polynomial of degree
three. This proof can be also found in \cite[p. 85]{Mor}. }
\end{rem}
From the above theorem we can prove easily that the answer for the
Question \ref{ques1} is positive.

\begin{cor}\label{cor1}
For any $N\in\mathbb{N}_{+}$, then there are infinitely many
polynomials $f\in\mathbb{Z}[X]$ of degree five with such a property
that on the hypersurface  $\cal{V}_{f}:\;f(p)+f(q)=f(r)+f(s)$ there
are at least $N$ nontrivial integer points.
\end{cor}
\begin{proof}
Let $b\neq 0$ and consider the polynomial $f(X)=X^5+aX^3+bX^2+cX$.
From the previous theorem we know that in this case the diophantine
equation $f(p)+f(q)=f(r)+f(s)$ has infinitely many solutions in
rational numbers. Let us take $N$ various rational solutions of our
equation, say
$(p_{i}/p_{i}',q_{i}/q_{i}',r_{i}/r_{i}',s_{i}/s_{i}')$ for
$i=1,2,\ldots,N$, and define
\begin{equation*}
d=\op{LCM}(p_{1}',q_{1}',r_{1}',s_{1}',\ldots,p_{N}',q_{N}',r_{N}',s_{N}').
\end{equation*}
If we now define $F(X)=X^5+ad^2X^3+bd^3X^2+cd^4X$, then on the
hypersurface  $\cal{V}_{f}:\; F(p)+F(q)=F(r)+F(s)$ we have the
points $(dp_{i}/p_{i}',dq_{i}/q_{i}',dr_{i}/r_{i}',ds_{i}/s_{i}')$
for $i=1,2,\ldots,N$, which are triplets of integers.
\end{proof}

The above corollary give us a positive answer for the Question
\ref{ques1}, but we see that if $N$ is growing then the coefficients
of the polynomial $F$ are growing too. So, we can state the
following:
\begin{ques}\label{ques3}
Let  $N>1$ be given. It is possible to construct a polynomial
$f(X)=X^5+aX^3+bX^2+cX$ with such a property that at least one
non-zero coefficient $a,b,c$ is independent of $N$ and for the set
of integer points on the hypersurface $\cal{V}_{f}$ we have $\sharp
(\cal{V}_{f}(\Z)\setminus T_{f})\geq N$?
\end{ques}
As we will see, the answer on this question is positive too. Before
we show how we can do that lets go back to the Question \ref{ques2}
in the case when $f$ is of the form $f(X)=X^5+aX^3+cX$.
Unfortunately, we are unable to prove theorem similar to the Theorem
\ref{thm1} in this case. However we can prove the following:

\begin{thm}\label{thm2}
Let $f(X)=X^5+aX^3+cX\in\Z[X]$. If $a<0$ and $a\not\equiv 2, 18,
34\pmod{48}$ then the diophantine equation $f(p)+f(q)=f(r)+f(s)$ has
two-parametric rational solution.
\end{thm}
\begin{proof}
For the proof we put
\begin{equation}\label{R4}
p=\frac{-x+y+3z}{5},\quad q=\frac{2x+y}{5},\quad
r=\frac{3y}{5},\quad s=\frac{x-y+3z}{5}.
\end{equation}
For $p,q,r,s$ defined in this way we have
\begin{align*}
f(p)+f(q)&-f(r)-f(s)=\\
         &\frac{6(x-y)(x+2y-3z)(x+2y+3z)(x^2+2y^2+3z^2+5a)}{625}.
\end{align*}
First three factors in the above identities lead to the trivial
solutions of our equation. So, we can see that the quantities given
by (\ref{R4}) lead to the nontrivial solution of the equation
$f(p)+f(q)=f(r)+f(s)$, if and only if $x^2+2y^2+3z^2+5a=0$. In
particular it must be $a<0$. As we know, local to global principle
of Hasse is true for diophantine equations of degree two. It means
that the diophantine equation $x^2+2y^2+3z^2+5a=0$ has solution in
rational numbers if and only if it has solutions in the field of
$p$-adic numbers  $\Q_{p}$ for any given
$p\in\mathbb{P}\cup\{\infty\}$, where as usual $\Q_{\infty}=\R$.

In the theorem below we find the well known algorithm of the
solvability of the diophantine equation of the form
$a_{1}X_{1}^2+a_{2}X_{2}^2+a_{3}X_{3}^2+a_{4}X_{4}^{2}=0$.

\begin{thm}
If
$f(x_{1},x_{2},x_{3},x_{4})=a_{1}X_{1}^2+a_{2}X_{2}^2+a_{3}X_{3}^2+a_{4}X_{4}^{2}$,
where $a_{i}\in\Z\setminus\{0\}$ are square-free and no three have a
factor in common, then $f$ represents zero if and only if the
following three conditions hold:
\begin{enumerate}
\item Not all coefficient have the same sign.

\item If $p$ is an odd prime dividing two coefficients and for which $(d/p^2|p)=1$, then
$(-a_{i}a_{j}|p)=1$, where $\op{GCD}(a_{i}a_{j},p)=1$ and
$d=a_{1}a_{2}a_{3}a_{4}$ is a discriminant of the form $f$.

\item If $d\equiv 1\pmod{8}$ or $d/4\equiv 1\pmod{8}$ then we have
$(-a_{1}a_{2},-a_{2}a_{3})_{2}=1.$
\end{enumerate}
\end{thm}

Proof of this theorem can be found in \cite{Jon}.

The quantity $(\alpha,\beta)_{2}$ takes two values: $+1$ or $-1$,
dependent on if the equation $\alpha x_{1}^{2}+\beta x_{2}^{2}=1$
has a solution in $\Q_{2}$. If $\alpha=2^{u}\alpha_{1},
\beta=2^{v}\beta_{1}$ and $\op{GCD}(2,\alpha_{1}\beta_{1})=1$, then
we have that
$(\alpha,\beta)_{2}=(2|\alpha_{1})^{v}(2|\beta_{1})^{u}(-1)^{(\alpha_{1}-1)(\beta_{1}-1)/4}$.
Here $(\cdot|\cdot)$ is an usual symbol of Legendre.

In order to finish the proof of our theorem we apply the above
procedure to the quadratic form
$X_{1}^2+2X_{2}^2+3X_{3}^2+5aX_{4}^{2}$. We must consider four cases
dependent on the values of $\op{GCD}(a,6)$. Because this reasoning
is very simple we leave it to the reader.
\end{proof}

\begin{exam}\label{exam1}
{\rm Let $f(X)=X^5-X^3+cX$ and consider the equation
$f(p)+f(q)=f(r)+f(s)$. We will show how we can use the previous
theorem in practice.

We consider the equation $(\ast)\;x^2+2y^2+3z^2-5=0$. This equation
has a rational solution $(x,y,z)=(0,1,1)$. Let us put $x=uT, y=vT+1,
z=T+1$ and next solve the equation $(uT)^2+2(vT+1)^2+3(T+1)^2-5=0$
in respect to $T$. After some necessary simplifications we get
parametrization of rational solutions of the equation  $(\ast)$ in
the form
\begin{equation*}
x=-\frac{2u(2v+3)}{u^2+2v^2+3},\quad
y=\frac{u^2-2v^2-6v+3}{u^2+2v^2+3},\quad
z=\frac{u^2+2v^2-4v-3}{u^2+2v^2+3}.
\end{equation*}

Using the parametrization we have obtained we get the solution of
the equation $f(p)+f(q)=f(r)+f(s)$, where $f(X)=X^5-X^3+cX$, in the
form
\begin{align*}
 &p=\frac{2(2u^2+(2v+3)u+2v^2-9v-3)}{5(u^2+2v^2+3)},\\
 &q=\frac{u^2-4(2v+3)u-2v^2-6v+3}{5(u^2+2v^2+3) }, \\
 &r=\frac{3(u^2-2v^2-6v+3)}{5(u^2+2v^2+3)},\\
 &s=\frac{2(u^2-(2v+3)u+4v^2-3(v +2))}{5(u^2+2v^2+3)}.
\end{align*}
}
\end{exam}

Using the method of proof of the Theorem \ref{thm2} we will show the
following:
\begin{cor}
Answer to the Question {\rm \ref{ques3}} is positive.
\end{cor}
\begin{proof}
This is a simple consequence of the fact that for any number $N$ we
can find a negative number $a_{N}$ such that the equation
$x^2+2y^2+3z^2=-5a_{N}$ has at least $N$ solutions in positive
integers $x,y,z$ all divisible by 5. In order to prove this let us
put $g_{N}=\prod_{k=1}^{N}(k^2+2)$ and $a_{N}=-(5g_{N})^2$. Next, we
define
\begin{equation*}
x_{k}=\frac{5g_{N}}{k^2+2}(2k+3),\quad
y_{k}=\frac{5g_{N}}{k^2+2}(k^2+3k-2),\quad
z_{k}=\frac{5g_{N}}{k^2+2}(k^2-2k-1),
\end{equation*}
for $k=1,2,\ldots,\;N.$ Note that the numbers $x_{k},y_{k},z_{k}$
are integers and are divisible by 5.

Due to the fact that
\begin{equation*}
\Big(\frac{2k+3}{k^2+2}\Big)^2+2\Big(\frac{k^2+3k-2}{k^2+2}\Big)^2+3\Big(\frac{k^2-2k-1}{k^2+2}\Big)^2=5,
\end{equation*}
we see that
\begin{equation*}
x_{k}^2+2y_{k}^2+3z_{k}^2=-5a_{N}\quad \mbox{for}\;k=1,2,\ldots,\;N.
\end{equation*}

Now define $f_{N}(x)=x^5+a_{N}x^3+cx$, where $c$ is an integer
number. From your reasoning we can see that on the hypersurface
$\cal{V}_{f_{N}}$ there is at last $N$ integer points given by
\begin{equation*}
p_{k}=\frac{-x_{k}+y_{k}+3z_{k}}{5},\quad
q_{k}=\frac{2x_{k}+y_{k}}{5},\quad r_{k}=\frac{3y_{k}}{5},\quad
s_{k}=\frac{x_{k}-y_{k}+3z_{k}}{5},
\end{equation*}
for $k=1,2,\ldots,\;N$ and number $c$ is independent of $N$.

\end{proof}

Results of this section suggest the following:
\begin{conj}\label{conj1}
Let $f(x)=x^5+ax^3+cx,$ where $a,c\in\Z$ not both zero and consider
the hypersurface $\cal{V}_{f}$. Then the set
$\cal{V}_{f}(\Q)\setminus T_{f}$ is infinite.
\end{conj}

\section{Construction of $\Q(i)$-rational points on $\cal{V}_{f}$}\label{Section3}

In this section we will consider the problem of the construction of
$\Q(i)$-rational points on the hypersurface $\cal{V}_{f}$.

Let us go back to the equation of the surface $\cal{S}$ from the
proof of the Theorem \ref{thm1} and let us note that for the degree
of the polynomial $\Delta$ we have $\op{deg}_{x}\Delta=2$. Now we
look on $\cal{S}$ as on a curve defined over field $\Q(i)(y)$, where
$i^2+1=0$. It is easy to see that $\cal{S}$ is a rational curve.
Indeed, on the curve $\cal{S}$ there is a $\Q(i)(y)$-rational point
$[x:v:w]=[i:10y:0]$ (it is a point at infinity). Putting
$x=ip,\;w=10yp+u$ and solving the obtained equation according to $p$
we get the parametrization of our curve given by
\begin{equation*}
x=-i\frac{u^2+75y^4+60ay^2+40by}{20y(u-5iy^2)},\quad
w=\frac{u^2-10iuy^2-75y^4-60ay^2-40by}{2(u-5iy^2)}.
\end{equation*}
Using the above parametrization we can find two-parametric solution
of the equation defining $\cal{V}_{f}$ in the form
\begin{align*}
&p=-i\frac{u^2+75y^4+60ay^2+40by}{20y(u-5iy^2)},\\
&q=i\frac{u^2-20iy^2u-25y^4+60ay^2+40by}{20y(u-5iy^2)},\\
&r=\frac{u^2+10(1-i)y^2u-25(3+2i)y^4-60ay^2-40by}{20y(u-5iy^2)},\\
&s=-\frac{u^2-10(1+i)y^2u-25(3-2i)y^4-60ay^2-40by}{20y(u-5iy^2)}.
\end{align*}

We sum up the above discussion concerning the existence of
$\Q(i)$-rational points on $\cal{V}_{f}$ in the following:

\begin{thm}\label{thm2}
Let $f(X)=X^5+aX^3+bX^2+cX\in\Z[X]$ and consider the hypersurface
$\cal{V}_{f}$. If $a=b=0$ then there exist $\Q(i)$-rational curve
contained in $\cal{V}_{f}$. If $a\neq 0$ or $b\neq 0$ then there
exist $\Q(i)$-rational surface contained in $\cal{V}_{f}$.
\end{thm}

Let us note that in the above quantities for $p,q,r,s$ the number
$c$ does not appear explicitly and that obtained solution is
non-trivial for any choice of $a,b,c\in\Z$. If we put $a=b=c=0$,
then we get a parametric solution (defined over $\Q(i)$) of the
diophantine equation $p^5+q^5=r^5+s^5$. After necessary
simplifications our solution in this case is of the form (in
homogenous form)
\begin{align*}
& p=u^2+75v^2,\\
& q=-u^2+20iuv+25v^2,\\
& r=iu^2+10(1+i)uv+25(2-3i)v^2,\\
& s=-iu^2-10(1-i)uv+25(2+3i)v^2.
\end{align*}

Probably this solution is well known but we cannot find it in the
literature of subject. We should note that this solution can be used
in the construction of a parametric solution (defined over $\Z[i]$)
of the diophantine equation
\begin{equation*}
p^{5n}+q^5=r^5+s^5,
\end{equation*}
where $n$ is a given positive integer. Indeed, it is easy to see
that the diophantine equation $u^2+75v^2=X^n$ has a parametric
solution given by the solution of the system of equations
\begin{equation*}
u+\sqrt{-75}v=(t_{1}+\sqrt{-75}t_{2})^n,\quad
u-\sqrt{-75}v=(t_{1}-\sqrt{-75}t_{2})^n,\quad X=t_{1}^2+75t_{2}^2.
\end{equation*}
It is clear that the solutions $u,\;v,\;X$ lead to the polynomial
solution of the equation $p^{5n}+q^5=r^5+s^5$.

\section{Possible generalizations of the results}

In this section we consider natural generalizations of the equation
defining the hypersurface $\cal{V}_{f}$ which has been considered in
the previous paragraphs.

First natural generalization which came to mind is considering the
following hypersurface
\begin{equation*}
\cal{V}_{F,\;G}:F(p)+G(q)=F(r)+G(s),
\end{equation*}
where $F(x)=x^5+ax^3+bx^2+cx,\;G(x)=x^5+dx^3+ex^2+fx$ and
$F(x)-F(0)\neq G(x)-G(0)$. It is clear that in order to find
rational points on the $\cal{V}_{F,\;G}$ we can assume that
$a,b,\ldots,e\in\Z$.

As we will see it is possible to show that for given $F,G$ satisfied
the above conditions the hypersurface $\cal{V}_{F,\;G}$ contain
elliptic surface defined over $\Q$. In order to show this let us
define
\begin{equation}\label{sub}
p=t-\frac{U}{V},\quad q=\frac{U}{V},\quad r=\frac{1}{V},\quad
s=t-\frac{1}{V}.
\end{equation}
For $p,q,r,s$ defined in this way we get
\begin{equation*}
F(p)+G(q)-F(r)-G(s)=- \frac{tV-U-1}{v^4}H(U,V,t),
\end{equation*}
where $H(U,V,t)=\sum_{i+j\leq 3}a_{i,\;j}U^{i}V^{j}$ and

\begin{equation}\label{ais}
\begin{array}{l}
a_{3,\;0}=-a_{2,\;0}=a_{1,\;0}=-a_{0,\;0}=5t,\\
a_{2,\;1}= -a+d-5t^2,\quad a_{1,\;1}=a-d,\\
a_{0,\;1}= -a+d+5t^2, \quad a_{1,\;2}=b+e+(2a+d)t+5t^3,\\
a_{0,\;2}=-b-e-(a+2d)t-5t^3,\quad a_{0,\;3}= f-c+(e-b)t+(d-a)t^2.
\end{array}
\end{equation}

Let us note that we can look on the surface $S_{F,\;G}: H(U,V,t)=0$
as on the cubic curve defined over the field $\Q(t)$. This has a
$\Q(t)$-rational point $P=(U,V)=(1,0)$. So, we can look on $P$ as on
the point at infinity and transform birationaly $S_{F,\;G}$ onto the
elliptic surface $\cal{E}_{F,\;G}$ with the Weierstrass equation of
the form
\begin{equation*}
\cal{E}_{F,\;G}: Y^2+a_{1}XY+a_{3}Y=X^3+a_{2}X^2+a_{4}X+a_{6},
\end{equation*}
where $a_{i}$ are certain polynomials in $\Z[t]$ depending on the
coefficients of polynomials $F, G$.

Although it is possible to give exact values of polynomials $a_{i}$
we do not give them here due to the fact that they are rather huge
polynomials. Instead we give some numerical results concerning the
existence of rational points on the hypersurface $\cal{V}_{F,\;G}$,
where $F(x)= x^5+cx,\;G(x)=x^5+fx$ and $c\neq f$.

In this case the surface $S_{F,\;G}$ takes the form
\begin{equation*}
S_{F,\;G}:\quad
5tU^3-5t^2U^2V+5t^3UV^2-(c-f)V^3-5tU^2-5t^3V^2+5tU+5t^2V-5t=0.
\end{equation*}
If we make a change of variables
\begin{align*}
(u,v,t)=\Big(&\frac{50t^4X^2+100t^2Y(c-f)-X^3}{50t^4(100(c-f)^2-10X(c-f)+X^2)-X^3},\\
             &\quad -\frac{10tX(Y-5t^2(10(c-f)-X))}{50t^4(100(c-f)^2-10X(c-f)+X^2)-x^3},\;t\Big),
\end{align*}
with the inverse
\begin{align*}
(X,Y,t)=\Big(&\frac{10(c-f)tV}{U-1},\\
           &\quad -\frac{10(c-f)t((c-f)V^3-5t^2(U-1)V(tV-U)-5t(U-1)^2U) }{(U-1)^2},\;t\Big)
\end{align*}
the surface $S_{F,\;G}$ is transformed to the form
\begin{equation*}
\cal{E}_{F,G}: Y^2=X^3-25t^4X^2-2500(c - f)^2t^4.
\end{equation*}

Unfortunately, we are unable to show that for any pair of integers
$c,f$ it is possible to find a rational number $t=t(c,f)$ with such
a property that the elliptic curve $\cal{E}_{t}:
Y^2=X^3-25t^4X^2-2500(c - f)^2t^4$ (which is a specialization of
surface $\cal{E}_{F,\;G}$ in $t$) has a positive rank. However, we
check that if $D=|c-f|\leq 10^3$ then there exists specialization of
$\cal{E}_{F,\;G}$ with positive rank. In the Table below we give
values for $t$ in the case $D\leq 100$.

\begin{center}\begin{tabular}{|c|l|l||c|l|l|}
  \hline
  $D$ & $t$ & Nontorsion point on $\cal{E}_{t}$ & $D$
  & $t$ & Nontorsion point on $\cal{E}_{t}$\\
  \hline

  1& 2/45&$( 1, 403/405)$& 51& 2/35&$( 17, 17051/245)$\\

  2& 2/45&$( 1, 397/405)$& 52& 1/15&$( 26/5, 598/225)$\\

  3& 2/7&$( 25, 6075/49)$& 53& 1/39&$( 25/16, 85925/97344)$\\

  4& 2/45&$( 4/5, 1208/2025)$& 54& 4/25&$( 17, 1429/125)$\\

  5& 1/15&$( 5/4, 61/72)$& 55& 1/95&$( 5/4, 3939/2888)$\\

  6& 1/5&$( 6, 42/5)$& 56& 1/15&$( 34/5, 2842/225)$\\

  7& 1/3&$( 14, 308/9)$& 57& 3/35&$( 38/5, 836/1225)$\\

  8& 2/15&$( 25, 1123/9)$& 58& 1/5&$( 29, 522/5)$\\

  9& 1/5&$( 13, 216/5)$& 59& 2/15&$( 25, 1021/9)$\\

 10& 1/3&$( 25, 1000/9)$& 60& 1/17&$( 25, 36000/289)$\\

 11& 1/15&$( 2, 64/45)$& 61& 2/15&$( 20, 640/9)$\\

 12& 2/11&$( 25, 14925/121)$& 62& 4/5&$( 164, 2232/5)$\\

 13& 7/11&$( 50, 25800/121)$& 63& 14/43&$( 49, 136857/1849)$\\

 14& 1/9&$( 25, 10100/81)$& 64& 4/45&$( 64/5, 77312/2025)$\\

 15& 1/9&$( 6, 308/27)$& 65& 2/5&$( 65, 39)$\\

 16& 4/45&$( 16, 25792/405)$& 66& 1/15&$( 6, 14/15)$\\

 17& 1/5&$( 17, 306/5)$& 67& 8/35&$( 32, 11136/245)$\\

 18& 2/5&$( 36, 792/5)$& 68& 1/15&$( 10, 250/9)$\\

 19& 1/5&$( 17, 294/5)$& 69& 1/35&$( 2, 64/245)$\\

 20& 1/3&$( 25, 500/9)$& 70& 1/9&$( 25, 9500/81)$\\

 21& 1/5&$( 42, 1344/5)$& 71& 4/141&$( 25, 2484475/19881)$\\

 22& 1/3&$( 25, 200/9)$& 72& 5/31&$( 25, 79500/961)$\\

 23& 4/21&$( 25, 51925/441)$& 73& 1/25&$( 73/20,
    19053/5000)$\\

 24& 3/25&$( 9, 2592/125)$& 74& 1/90&$( 37/40, 49469/64800)$\\

 25& 10/51&$( 25, 299875/2601)$& 75& 2/11&$( 25, 1875/121)$\\

 26& 11/15&$( 169, 91468/45)$& 76& 3/29&$( 25, 99400/841)$\\

 27& 8/25&$( 32, 14464/125)$& 77& 1/15&$( 34, 8888/45)$\\

 28& 1/9&$( 29/4, 5873/648)$& 78& 1/5&$( 29, 22/5)$\\

 29& 1/9&$( 50, 28600/81)$& 79& 2&$( 2084, 84048)$\\

 30& 2/7&$( 25, 1125/49)$& 80& 2/15&$( 20, 488/9)$\\

 31& 6/35&$( 25, 5701/49)$& 81& 2/15&$( 81, 3627/5)$\\

 32& 4/45&$( 16, 25408/405)$& 82& 2/171&$( 1, 24209/29241)$\\

 33& 1/5&$( 22, 396/5)$& 83& 1/9&$( 34, 15512/81)$\\

 34& 4/21&$( 16, 7424/441)$& 84& 1/35&$( 42/5, 4218/175)$\\

 35& 1/5&$( 17, 6/5)$& 85& 2/165&$( 1, 4253/5445)$\\

 36& 1/5&$( 18, 126/5)$& 86& 2/3&$( 344, 54352/9)$\\

 37& 1/65&$( 82/65, 74016/54925)$& 87& 1/15&$( 5597/324,
    2016283/29160)$\\

 38& 6/35&$( 25, 5477/49)$& 88& 1/61&$( 61/36,
    1496489/803736)$\\

 39& 1/45&$( 6/5, 604/675)$& 89& 1&$( 8381, 766104)$\\

 40& 3/13&$( 25, 11000/169)$& 90& 5/31&$( 25, 42000/961)$\\

 41& 1/5&$( 125/4, 1233/8)$& 91& 1/85&$( 26, 191568/1445)$\\

 42& 1/13&$( 6, 102/13)$& 92& 2&$( 2516, 114264)$\\

 43& 1/45&$( 10, 2560/81)$& 93& 1/22&$( 93/8, 148893/3872)$\\

 44& 5/21&$( 25, 2000/441)$& 94& 2/5&$( 89, 1833/5)$\\

 45& 1/34&$( 25/16, 375/2312)$& 95& 1/5&$( 38, 684/5)$\\

 46& 1/15&$( 46/5, 5842/225)$& 96& 5/31&$( 25, 4500/961)$\\

 47& 12/53&$( 25, 91925/2809)$& 97& 4/15&$( 97, 40061/45)$\\

 48& 1/34&$( 17/8, 21261/9248)$& 98& 1/3&$( 6125/81, 267050/729)$\\

 49& 1/3&$( 49, 1862/9)$& 99& 1/7&$( 50, 16600/49)$\\

 50& 1/45&$( 13/4, 18557/3240)$& 100& 2/15&$( 20, 88/9)$\\
  \hline
\end{tabular}
\end{center}

Our computations suggest the following
\begin{conj}\label{conj2}
Let $c,f\in\Z$ and $c\neq f$ and let us consider the elliptic
surface $\cal{E}:\; Y^2=X^3-25t^4X^2-2500(c - f)^2t^4.$ Then the set
\begin{equation*}
S=\{t\in\Q:\; \mbox{curve}\; \cal{E}_{t}\;\mbox{is an elliptic curve
and has a positive rank}\}
\end{equation*}
is nonempty.
\end{conj}

We firmly believe that the following conjecture is also true.
\begin{conj}\label{conj3}
Let $a,b,c,d,e,f\in\Z$ and let us consider the elliptic surface
\begin{equation*}
\cal{E}:\; H(U,V,t)=\sum_{i+j\leq 3}a_{i,\;j}U^{i}V^{j}=0
\end{equation*}
where $a_{i,j}$ are given by {\rm (\ref{ais})}. Then the set
\begin{equation*}
S=\{t\in\Q:\; \mbox{curve}\; \cal{E}_{t}\;\mbox{is an elliptic curve
and has a positive rank}\}
\end{equation*}
is nonempty.
\end{conj}

\bigskip

Second generalization which came to mind is considering the
following hypersurface
\begin{equation*}
\cal{V}^{f}:f(p,q)=f(r,s),
\end{equation*}
where $f$ is a symmetric quintic polynomial (i.e. $f(x,y)=f(y,x)$),
so a polynomial of the form
\begin{equation}\label{poly}
f(x,y)=\sum_{i=1}^{5}a_{i}(x^{i}+y^{i})+xy\sum_{i=1}^{3}b_{i}(x^i+y^i)+x^2y^2(c_{0}(x+y)+c_{1}).
\end{equation}

As we will see it is quite easy task to show that on the
$\cal{V}^{f}$ there is in general infinitely many $\Q(i)$-rational
points. In order to prove this we use the substitution given by
(\ref{sub}) and get a quadratic curve, say $C$, defined over $\Q(t)$
with $\Q(i)$-rational points and thus $\Q(i)$-rational curve.

Using now substitution given by (\ref{sub}) we get
\begin{equation*}
f(p,q)-f(r,s)=-\frac{(U-1)(tV-U-1)}{V^4}G(U,V,t),
\end{equation*}
where $G(U,V,t)=\sum_{i+j\leq 2}b_{i,j}u^iv^j$ and
\begin{equation*}
\begin{array}{lll}
  b_{2,0}=2a_{4}-2b_{2}+c_{1}+t(5a_{5}-3b_{3}+c_{0}), & b_{1,0}=0, & b_{0,0}=b_{2,0}, \\
  b_{0,2}=2a_{2}+t(3a_{3}-b_{1})+t^2(4a_{4}-b_{2})+t^3(5a_{5}-b_{3}), &  b_{0,1}=b_{1,1}, &
  b_{1,1}=-tb_{2,0}.
\end{array}
\end{equation*}

In order to construct $\Q(i)$-rational points on $\cal{V}^{f}$ we
must consider the quadratic $C:\;G(U,V,t)=0$. Note that
$G(i,0,t)=0$, so we can use standard method to parametrize
$\Q(i)$-rational points on $C$ and in general we get two parametric
solution of the equation $G(U,V,t)=0$. This implies the existence of
two parametric solution of the equation defining the hypersurface
$\cal{V}^{f}$.

It is clear that this method cannot be used always. Indeed, if
$b_{2,0}\equiv 0\in\Z[t]$ then the equation $G(U,V,t)=0$ is reduced
to the equation $b_{0,2}=0$ which has at most three solutions in
$\Q(i)$. However, if $b_{2,0}\neq 0$ for some $t$ then the curve $C$
is nontrivial and we can apply our method in order to construct
$\Q(i)$-rational points on $\cal{V}^{f}$. This suggests the
following

\begin{ques}\label{ques4}
Let us consider the hypersurface $\cal{V}_{f}:f(p,q)=f(r,s)$ where
$f$ is of the form {\rm (\ref{poly})}. Suppose that
$2a_{4}-2b_{2}+c_{1}=5a_{5}-3b_{3}+c_{0}=0$. Is it possible to
construct  $\Q(i)$-rational points on $\cal{V}^{f}$?
\end{ques}

Although it is possible, we do not give here equations defining the
parametrization of the curve $C$ in the case when
$b_{2,0}\in\Z[t]\setminus\{0\}.$ However, in the following example
we will show how this procedure works in practice.

\begin{exam}
{\rm Consider the following polynomial
\begin{equation*}
f(x,y)=x^5+y^5-5xy(x^2+y^2)+5xy(x+y)+5(x^2+y^2)-5(x+y).
\end{equation*}

It should be noted that hypersurface $\cal{V}^{f}$ for this
polynomial was considered by Consani and Scholten in \cite{ConScho}.

In this case we have that $C:\;U^2-tUV+(1-t+t^2)V^2-tV+1=0$. Using
the $\Q(i)$-rational point $P=(i,0)$ we get parametrization of $C$
in the form
\begin{equation*}
U=\frac{-iu^2-tu+i(t^2-t+1)}{u^2-tu+t^2-t+1},\quad
V=\frac{(1+i)(t-(1+i)u)}{u^2-tu+t^2-t+1}.
\end{equation*}
Using obtained parametrization of $C$ we get two parametric solution
of the equation defining $\cal{V}^{f}$ in the form
\begin{equation*}
p=\frac{(1+i)(u^2-(2-i)tu-it^2+t-1)}{2(t-(1+i)u)},\quad
q=\frac{(1+i)(-u^2-itu+t^2-t+1)}{2(t-(1+i)u)},
\end{equation*}
\begin{equation*}
r=\frac{(1+i)(iu^2-(2+i)tu+t^2-it+i)}{2(t-(1+i)u)},\quad
s=\frac{(1-i)(u^2-tu+t^2-t+1)}{2(t-(1+i)u)}.
\end{equation*}

}
\end{exam}

\bigskip

We should note that for specific choice of numbers
$a_{i},\;b_{j},\;c_{k}$ it is very likely that there is a rational
number $t_{0}$ with such a property that the quadric
$C_{t_{0}}:\;G(U,V,t_{0})=0$ has a rational point. Then we can use
standard method of parametrization of quadrics in order to get
rational solutions of the equation defining the hypersurface
$\cal{V}^{f}$. We illustrated this procedure with an example.

\begin{exam}{\rm Consider the polynomial $f$ given by (\ref{poly})
and suppose that $c_{1}=29$ and all remaining coefficients in the
polynomial $f$ are equal to one. Let us put $t=2$. Then on the curve
$C_{2}:\;32U^2-32UV+11V^2-32V+32=0$ we have rational point $P=(3,8)$
and parametrization of the curve $C_{2}$ is of the form
\begin{equation*}
U=\frac{160w^2-144w+33}{32w^2-32w+11},\quad
V=\frac{8(32w^2-24w+5)}{32w^2-32w+11}.
\end{equation*}
Finally, we get rational curve on the hypersurface
$\cal{V}^{f}:f(p,q)=f(r,s)$ of the form
\begin{equation*}
p=\frac{96w^2-48w+7}{8(32w^2-24w+5)},\quad
q=\frac{160w^2-144w+33}{8(32w^2-24w+5)},
\end{equation*}
\begin{equation*}
r=\frac{224w^2-160w+29}{8(32w^2-24w+5)},\quad
s=\frac{32w^2-32w+11}{8(32w^2-24w+5)}.
\end{equation*}
}
\end{exam}

\bigskip

 \hskip 4.5cm       Maciej Ulas

 \hskip 4.5cm       Jagiellonian University

 \hskip 4.5cm       Institute of Mathematics

 \hskip 4.5cm       Reymonta 4

 \hskip 4.5cm       30 - 059 Krak\'{o}w, Poland

 \hskip 4.5cm      e-mail:\;{\tt Maciej.Ulas@im.uj.edu.pl}

 \end{document}